\documentclass{llncs}
\usepackage{graphicx,amssymb,amsmath,llncsdoc}

\title{Stochastic Flips on Dimer Tilings}

\author{Thomas Fernique \and Damien Regnault}

\institute{
LIF, CNRS \& Univ. de Provence\\
39 rue Joliot-Curie 13453 Marseille -- France\\
\email{\{thomas.fernique,damien.regnault\}@ens-lyon.org}
}

\begin{document}

\maketitle

\begin{abstract}
This paper introduces a Markov process inspired by the problem of quasicrystal growth.
It acts over dimer tilings of the triangular grid by randomly performing local transformations, called {\em flips}, which do not increase the number of identical adjacent tiles (this number can be thought as the tiling energy).
Fixed-points of such a process play the role of quasicrystals.
We are here interested in the worst-case expected number of flips to converge towards a fixed-point.
Numerical experiments suggest a $\Theta(n^2)$ bound, where $n$ is the number of tiles of the tiling.
We prove a $O(n^{2.5})$ upper bound and discuss the gap between this bound and the previous one.
We also briefly discuss the average-case.
\end{abstract}

\section*{Introduction}

Tilings are often used as a toy model for quasicrystals, with minimal energy tilings being characterized by local properties called \emph{matching rules}.
In this context, a challenging problem is to provide a theory for quasicrystals growth.
One of the proposed theories relies on a relaxation process (\cite{janot} p. 356): a tiling with many mismatches is progressively corrected by local transformations called \emph{flips}.
Ideally, the tiling eventually satisfies all the matching rules and thus shows a quasicrystalline structure.
It is compatible with experiments, where quasicrystals are obtained from a hot melt by a slow cooling during which flips really occur (Bridgman-Stockbarger method).
It is however unclear whether only flips can explain successful coolings or if other mechanisms should be taken into account.
This is deeply related with the convergence rate of such a flip-correcting process.\\

A cooling process aiming to be physically realist is described in \cite{aperiodic}.
It considers so-called {\em cut and project tilings} of any dimension and codimension, and performs flips which modify by $\Delta E$ the energy of the tiling with a probability proportional to $\exp(-\Delta E/T)$, so that the stationary distribution at fixed temperature $T$ is the Boltzmann one.\\

A simplified cooling process is obtained by performing equipro\-ba\-bly at random only flips whose corresponding $\Delta E$ is above a fixed threshold.
It has been studied on tilings of dimension one and codimension one (two-letter words) in \cite{analco}.
We here focus on tilings of dimension two and codimension one (dimer tilings).\\

The paper is organized as follows.
Sec.~\ref{sec:settings} introduce notations and basic definitions.
We then describe, in Sec.~\ref{sec:convergence}, the cooling process we consider.
We also formally state the convergence time we want to bound, and make conjectures based on numerical experiments.
Sec.~\ref{sec:bound} is then devoted to the proof of an upper bound on the convergence time.
This proof relies on a concentration result for some well-chosen function.
We conclude the paper by a short section discussing prospects of this work, namely the non-tightness of the obtained theoretical bound, and the average convergence time, {\em i.e.}, when the initial tiling is chosen at random (instead of considering the one with the greatest convergence time).

\section{Settings}
\label{sec:settings}

\noindent {\bf Dimer tiling}\\
Let $(\vec{v}_1,\vec{v}_2,\vec{v}_3)$ be the unit vectors of the Euclidean plane of direction $\frac{\pi}{6}+k\frac{2\pi}{3}$, $k=1,2,3$.
They generate the so-called triangular grid.
A {\em dimer} is a lozenge tile made of two adjacent triangles of the grid.
A {\em domain} is a connected subset of the grid.
Then, a {\em dimer tiling} is a tiling  of a domain by dimers (Fig.~\ref{fig:dimer_tiling}, left).\\

\noindent {\bf Lift}\\
Let $(\vec{e}_1,\vec{e}_2,\vec{e}_3)$ be the canonical basis of the Euclidean space.
The {\em lift} of a dimer tiling is its image by a map $\phi$ which is linear over tiles and satisfies $\phi(\vec{x}+\vec{v}_k)=\phi(\vec{x})+\vec{e}_k$ for any two vertices $\vec{x}$ and $\vec{x}+\vec{v}_k$ connected by an edge of the tiling.
This map is uniquely defined up to a translation; we assume that $(0,0,0)$ belongs to the image of the domain boundary.
Tiles are thus mapped onto facets of three-dimensional unit cubes, and dimer tilings are mapped onto stepped surfaces of $\mathbb{R}^3$.
This can be easily seen by shading tiles (Fig.~\ref{fig:dimer_tiling}, center).\\

\noindent {\bf Height}\\
Following \cite{thurston}, we define the {\em height} of a point in a dimer tiling by the distance of its image under $\phi$ to the plane $x+y+z=0$.
We also define the {\em height} of a tile as the height of its center: this yields a third representation of dimer tilings, where tile colors depend on the height (Fig.~\ref{fig:dimer_tiling}, right).\\

\begin{figure}
\includegraphics[width=\textwidth]{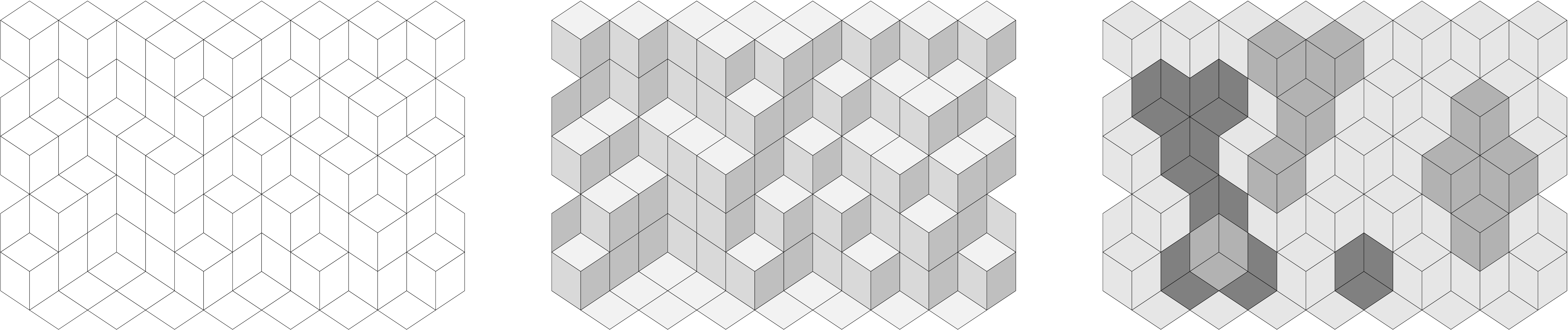}
\caption{Three representations of the same dimer tiling.}
\label{fig:dimer_tiling}
\end{figure}

\noindent {\bf Errors and energy}\\
An {\em error} in a dimer tiling is an edge shared by two identical tiles (up to a translation).
Error-free regions thus perfectly alternate tiles.
The corresponding subsets of the triangular grid can be tiled by six-triangle hexagons, and the lift of an error-free region approximate the plane $x+y+z=0$
In particular, all the tiles have the same height.
The {\em energy} $E(\omega)$ of a dimer tiling $\omega$ is its total number of errors.\\

\noindent {\bf Islands and holes}\\
Suppose that a dimer tiling $\omega$ contains a finite, connected and simply connected set of tiles $\sigma$ whose boundary edges are errors and whose domain can be tiled by six-triangle hexagons.
The number $A(\sigma)$ of such hexagons is called the {\em area} of $\sigma$.
One checks that the height of tiles with an edge on the boundary of $\sigma$ take only two values, say $h(\sigma)$ and $h(\overline{\sigma})$, respectively depending whether the tile is in $\sigma$ or not.
One says that $\sigma$ is an {\em island} if $0<h(\overline{\sigma})<h(\sigma)$, a {\em hole} otherwise.
In both cases, $h(\sigma)$ is called the {\em height} of $\sigma$.
Fig.~\ref{fig:islands} illustrates this.\\

\begin{figure}
\includegraphics[width=\textwidth]{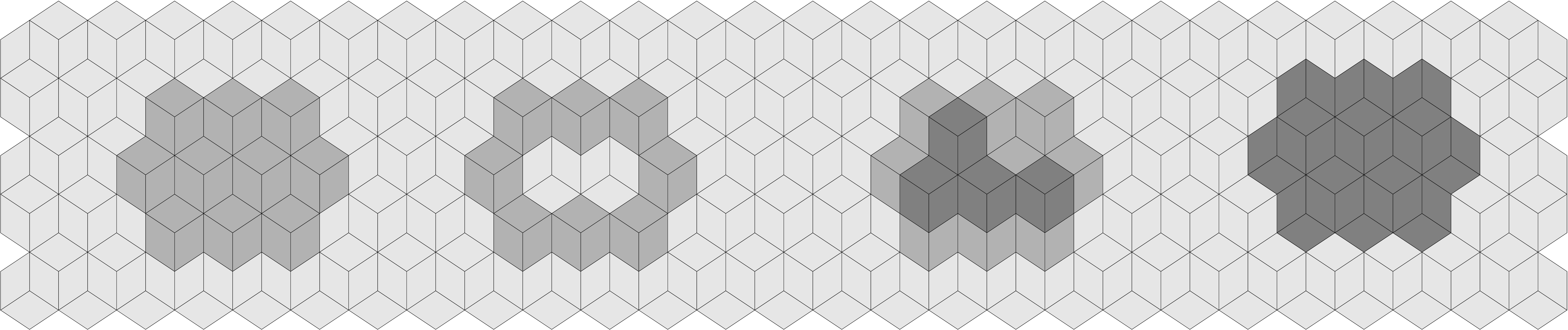}
\caption{
Boundary tiles of this dimer tiling have the same height, say $0$.
From left to right: an island of height $1$ and area $10$; the same island with a hole of height $0$ and area $2$; once again the same island with a superposed island of height $2$ and area $4$; an island of height $-1$ and area $10$ (which looks like a hole because its height is negative).}
\label{fig:islands}
\end{figure}

\noindent {\bf Volume}\\
The {\em volume} $V(\omega)$ of a dimer tiling $\omega$ is defined as the sum of areas of its islands minus the sum of areas of its holes.
Error-free tilings thus have volume zero.
One checks that if the domain is finite, simply connected and admits an error-free tiling, then the errors of any dimer tiling of this domain are the boundaries of its islands and holes.
Hence, tilings of volume zero are error-free.
Otherwise, non-closed paths of errors can run across the domain, with endpoints (if any) on the boundary of the domain.\\

\noindent {\bf Partial order}\\
A dimer tiling $\omega$ is said to be smaller than or equal to a dimer tiling $\omega'$ of the same domain, written $\omega\leq \omega'$, if each point has, in modulus, a height in $\omega$ smaller than or equal to its height in $\omega'$.
The set of tiligs of a domain becomes a distributive lattice, whose extremal elements are the tilings of extremal volume.

\noindent {\bf Flip}\\
Whenever a vertex $x$ of a dimer tiling belongs to exactly three tiles, one gets a new dimer tiling by translating each of these tiles along the edges shared by the two other ones (or, equivalently, by rotating these tiles by $\frac{\pi}{3}$ around $x$).
This local rearrangement of tiles is called a {\em flip} in $x$ (Fig.~\ref{fig:flips}).
In the lift, it corresponds to add or remove one cube.
One checks that a flip modifies the volume of a tiling by $\Delta V=\pm 1$ and its energy by $\Delta E\in\{0,\pm 2,\pm 4,\pm 6\}$.
The total number of flips which can be performed on a tiling $\omega$ is denoted by $F(\omega)$.\\

\begin{figure}
\includegraphics[width=\textwidth]{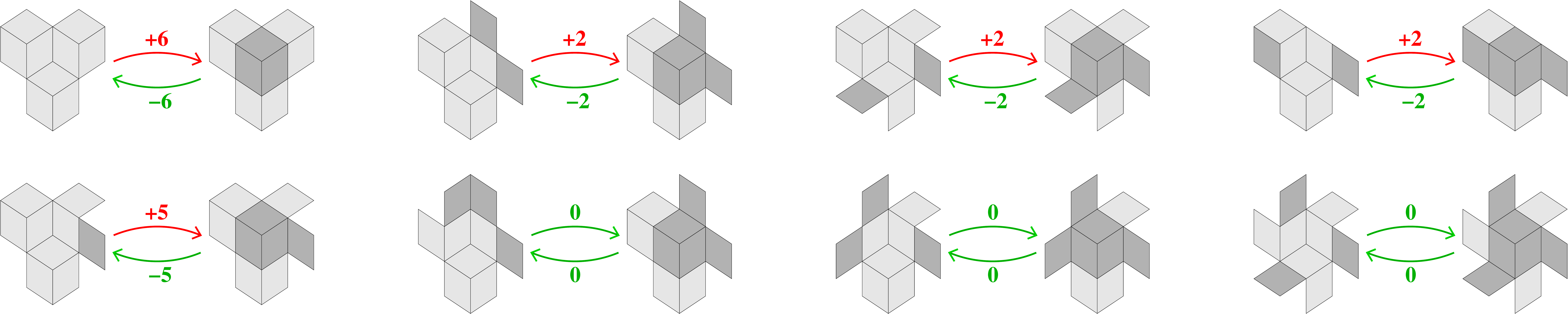}
\caption{A flip and all its possible immediate neighborhoods (up to a rotation), with the according variation of energy of the tiling being reported.
The four leftmost cases correspond to merging of island or creation of holes.
In the cooling process considered in Sec.~\ref{sec:convergence}, only flips which do not increase the energy will be allowed.
}
\label{fig:flips}
\end{figure}

Note that the volume of a tiling turns out to be the minimal number of flips required to transform this tiling into an error-free tiling (if the domain admits such a tiling).\\

\noindent {\bf Flip-accessibility}\\
In \cite{thurston}, Thurston has shown that any two dimer tilings can be connected by performing a sequence of flips, \emph{i.e.} are mutually {\em flip-accessible}, if they tile the same finite and simply connected subset of the triangular grid.
The case of dimer tilings of the whole triangular grid has been considered in \cite{bfr}.
Here, the physical motivations exposed in the introduction lead us to focus on a {\em constrained} flip-accessibility, where only flips which do not increase the number of errors are allowed.
Moreover, we are not interested in a flip-accessibility where flips can be carefully chosen, but in the case where flips are {\em randomly} performed.
The next section makes this more precise.

\section{The cooling process}
\label{sec:convergence}

Let us fix a finite, connected and simply connected subset of the triangular grid which does admit an error-free dimer tiling made of $n$ tiles.
Let $\Omega_n$ denotes the set of all the possible dimer tilings of this subset (flip-accessibility yields that all these tilings have $n$ tiles).\\

The so-called {\em cooling process} we are interested in is the Markov chain defined over $\Omega_n$ as follows.
It starts from $\omega_0\in\Omega_n$ and produces a sequence $(\omega_t)_{t=1,2,\ldots}$ where $\omega_{t+1}$ is obtained by performing on $\omega_t$ a flip, uniformly chosen at random among the flips which do not increase the number of errors\footnote{In other words, the $\Delta E$ threshold discussed in the introduction is equal to zero.}.
If there is no such flip, then the process stops (the tiling is said to be {\em frozen}).\\

Fig.~\ref{fig:average_cooling_snapshots} shows the cooling of a tiling of a hexagonal-shaped domain of $3779$ tiles, chosen uniformly at random among the tilings of this domain.
In this particular case, the process stops in $4290$ steps, with the obtained frozen tiling turning out to be error-free.
How general is such an evolution?


\begin{figure}
\centering
\includegraphics[width=\textwidth]{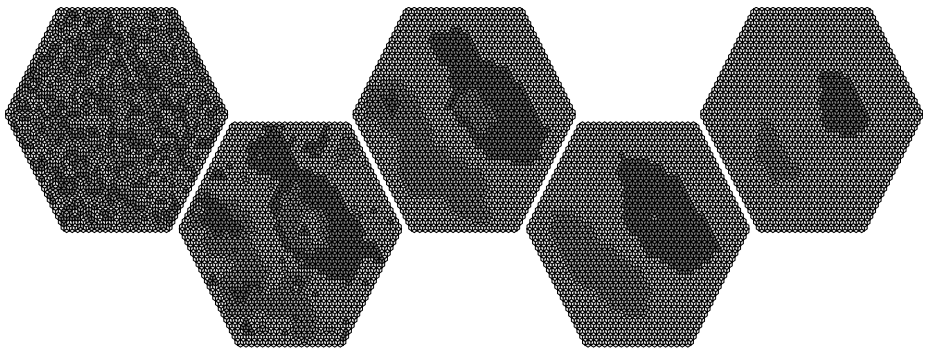}
\caption{Snapshots each $1000$ flips of the cooling of a random tiling (from left to right).}
\label{fig:average_cooling_snapshots}
\end{figure}

\noindent The following proposition will give us a first insight into this evolution:

\begin{proposition}\label{prop:possible_flip}
Whenever a tiling in $\Omega_n$ has an error, a flip which do not increase the number of errors can be performed onto.
Moreover, this flip can be chosen so that it decreases the volume of the tiling.
\end{proposition}

\begin{proof}
The existence of a (finite) error-free tiling ensures that errors form the boundary of islands.
Consider an island whose height is, in modulus, maximal.
Such an island is made of hexagons tiled by three-tiles.
Performing a flip on such a hexagon decreases the area of the island (hence the volume of the tiling) and decreases the number of errors by $2(k-3)$, where $k$ is the number of error edges of this hexagon.
It thus suffices to find a hexagon with at least three error edges.
Let us follow clockwise the boundary of the island: the direction between two consecutive edges changes by $\pm\frac{\pi}{6}$, and there is more negative variations because the boundary is closed.
We can thus find two consecutive negative variations: the three corresponding edges are errors on the boundary of the wanted hexagon.
\qed\end{proof}

This proposition ensures that a frozen tiling is necessarily error-free.
Moreover, it shows that the process almost surely stops in a finite time, since the volume of tilings in $\Omega_n$ is uniformly bounded and decreases with probability at least $\frac{1}{n}$ at each step.\\

In order to describe more precisely this cooling process, we are interested in the probability distribution of the random variable $T$, called the {\em convergence time}, which counts the number of steps before it stops.
Here, we focus only on the \emph{worst expected convergence time} $\widehat{T}$, defined by
$$
\widehat{T}(n):=\max_{\omega\in \Omega_n} \mathbb{E}(T~|~\omega_0=\omega),
$$
\noindent The fact that the cooling process almost surely stops in a finite time now reads
$$
\widehat{T}(n)<\infty.
$$
Conversely, since the volume of a tiling decreases at most by one at each step, the worst expected convergence time is bounded below by the maximal volume of tilings in $\Omega_n$.
With the hexagonal-shaped chosen domain, one checks that it yields the lower bound\footnote{$O$, $\Omega$ and $\Theta$ are the usual Bachmann-Landau notations.}
$$
\widehat{T}(n)=\Omega(n\sqrt{n}).
$$
Our goal is to obtain a theoretical tight bound for the worst expected convergence time.
Numerical experiments suggest that it is quadratic and correspond to tilings of maximal volume (Fig.~\ref{fig:simulations}).
We thus conjecture:
$$
\widehat{T}(n)=\Theta(n^2).
$$

\begin{figure}
\includegraphics[width=0.48\textwidth]{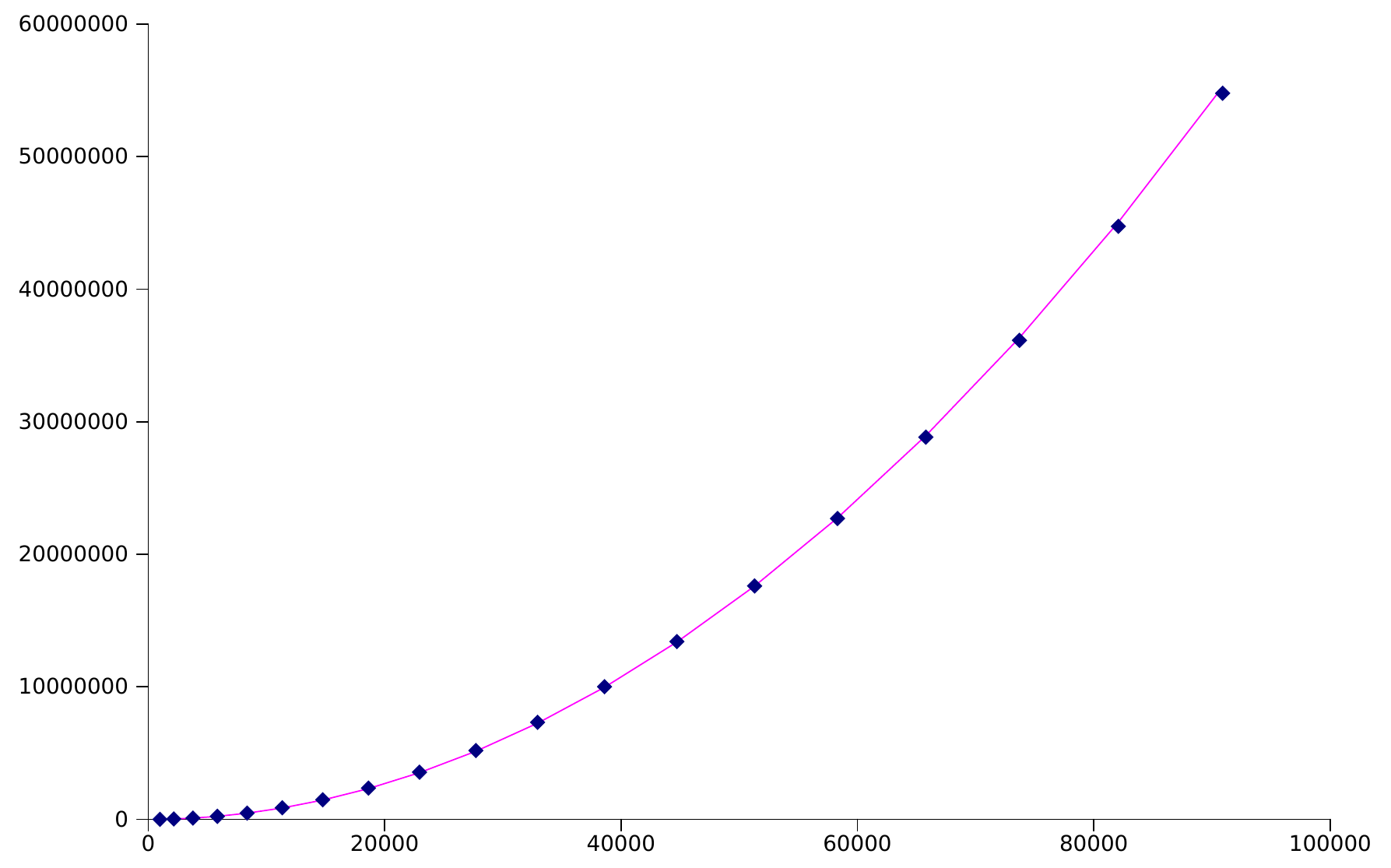}
\hfill
\includegraphics[width=0.48\textwidth]{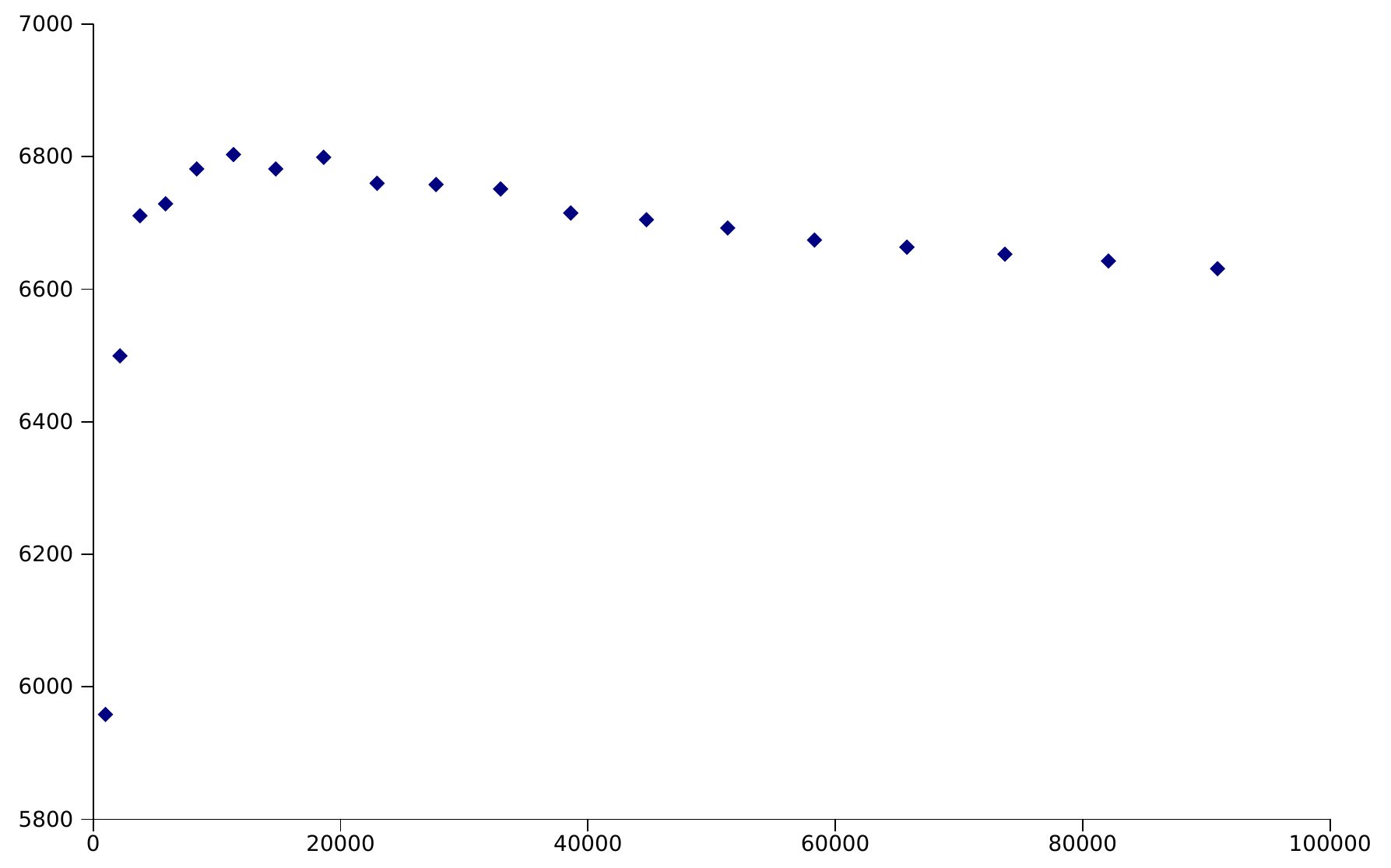}
\caption{Left, markers show averages of the convergence time on numerical experiments, starting from the maximal volume tiling (thought to have the worst convergence time).
We also drawn the curve $cn^2$, $c=6.689.10^{-3}$, which matches rather well the markers.
Right, dividing these experimental values by $n^2$ does not show a subpolynomial factor.
We thus conjecture $\widehat{T}(n)=\Theta(n^2)$.
}
\label{fig:simulations}
\end{figure}

\section{A theoretical upper bound}
\label{sec:bound}

To obtain an upper theoretical bound on the expected convergence time, we will rely on the following probabilistic tool (proven, {\em e.g.}, in \cite{fmst}):

\begin{proposition}\label{prop:probabilistic_tool1}
Let $(x_t)_{t\geq 0}$ be a Markov chain over a space $\Omega$.
Assume that there is $\varepsilon>0$ and a map $\phi:\Omega\to[a,b]\subset(0,\infty)$ such that, whenever $\phi(x_t)>a$:
$$
\mathbb{E}[\phi(x_{t+1})-\phi(x_t)|x_t]\leq -\varepsilon.
$$
Then, the expected value of the random variable $T:=\min\{t~|~\phi(x_t)= a\}$ satisfies
$$
\mathbb{E}(T)\leq \frac{b-a}{\varepsilon}.
$$
\end{proposition}

\noindent We first need to introduce the notion of {\em triconvexity} (see Fig.~\ref{fig:triconvexity}):

\begin{definition}
A dimer tiling is said to be {\em triconvex} if any segment of slope $0\mod \frac{2\pi}{3}$ which connects two vertices $x$ and $y$ contains only vertices whose height is between the heights of $x$ and $y$.
The {\em triconvex hull} $\overline{\omega}$ of a dimer tiling $\omega$ is the smallest triconvex dimer tiling greater than or equal to $\omega$.
\end{definition}

\begin{figure}
\includegraphics[width=\textwidth]{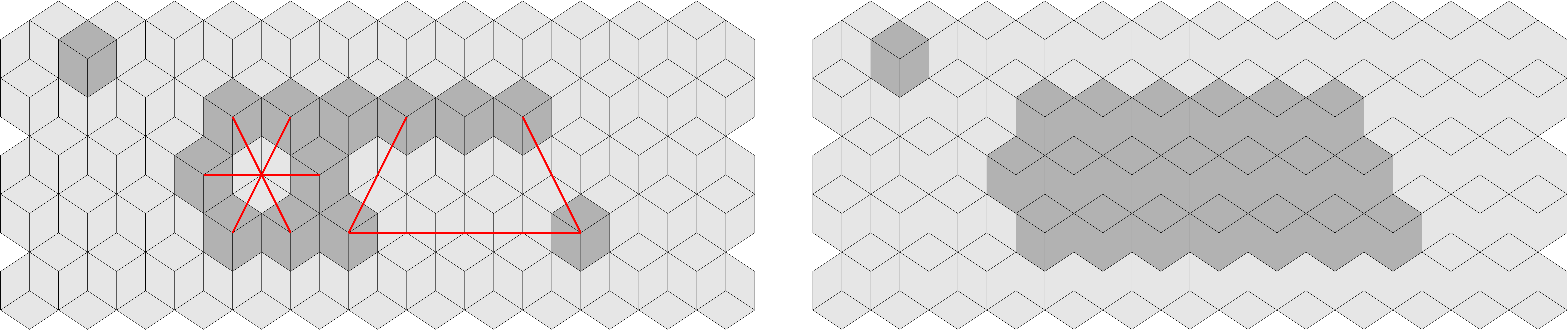}
\caption{A non-triconvex tiling (left, with fault segments) and its triconvex hull (right).}
\label{fig:triconvexity}
\end{figure}

Let us then associate with any map $X$ defined over dimer tilings (for example, the volume $V$ or the energy $E$) a map $\overline{X}$ defined as follows:
$$
\overline{X}(\omega):=X(\overline{\omega}).
$$
This will help us to avoid ambiguities, especially when considering variations.
Indeed, let us stress that $\Delta \overline{X}(\omega)$ and $\Delta X(\overline{\omega})$ are generally different, since
\begin{itemize}
\item $\Delta \overline{X}(\omega)=X(\overline{\omega_a})-X(\overline{\omega})$, where $\omega_a$ is obtained by performing a flip on $\omega$;
\item $\Delta X(\overline{\omega})=X(\omega_b)-X(\overline{\omega})$, where $\omega_b$ is obtained by performing a flip on $\overline{\omega}$.
\end{itemize}

\noindent Now, in order to use Prop.~\ref{prop:probabilistic_tool1}, let us define a map $\phi$ over dimer tilings by:
$$
\phi:=4V+E
$$
We first consider the most simple case:

\begin{lemma}\label{lem:une_ile_convexe}
If $\omega$ is a triconvex dimer tiling with only one island, then
$$
\mathbb{E}[\Delta \phi(\omega)]\leq-\frac{12}{F(\omega)}.
$$
\end{lemma}

\begin{proof}
Follow clockwise the boundary of the island: two consecutive edges make either a {\em salient} or a {\em reflex} angle, depending on whether we turn leftwards or rightwards (see Fig.~\ref{fig:une_ile_convexe}).
If there is a flip around some vertex $x$, then the six-triangle hexagon of center $x$ has at least $3$ error edges on its boundary.
If these error edges are consecutive, then they form $i\geq 2$ consecutive similar angles, and performing the flip yields
\begin{itemize}
\item if the $i$ angles are salient, then $\Delta V=-1$ and $\Delta E=4-2i$, that is, $\Delta \phi=-2i$.
This yields $\Delta\phi=-2$ per salient angle.
\item if the $i$ angles are reflex, then $\Delta V=1$ and $\Delta E=4-2i$, that is, \mbox{$\Delta \phi=8-2i$}.
This yields $\Delta\phi\leq +2$ per reflex angle.
\end{itemize}
If there are only such flips, then the claimed bound follows by checking by induction on its length that the boundary of an island (triconvex or not) has always six more salient than reflex angles.
Otherwise, consider a flip in $x$ such that the error edges on the boundary of the six-triangle hexagon of center $x$ are non consecutive.
At least two of these edges must be parallel.
Thus, by triconvexity, the vertex $x$ is in $\omega$.
Hence, such a flip decreases by $1$ the volume (and disconnects the island).
In particular, it decreases $\phi$ by at least $4$, so that it can only improve the bound.
\qed\end{proof}

\begin{figure}
\centering
\includegraphics[scale=0.2]{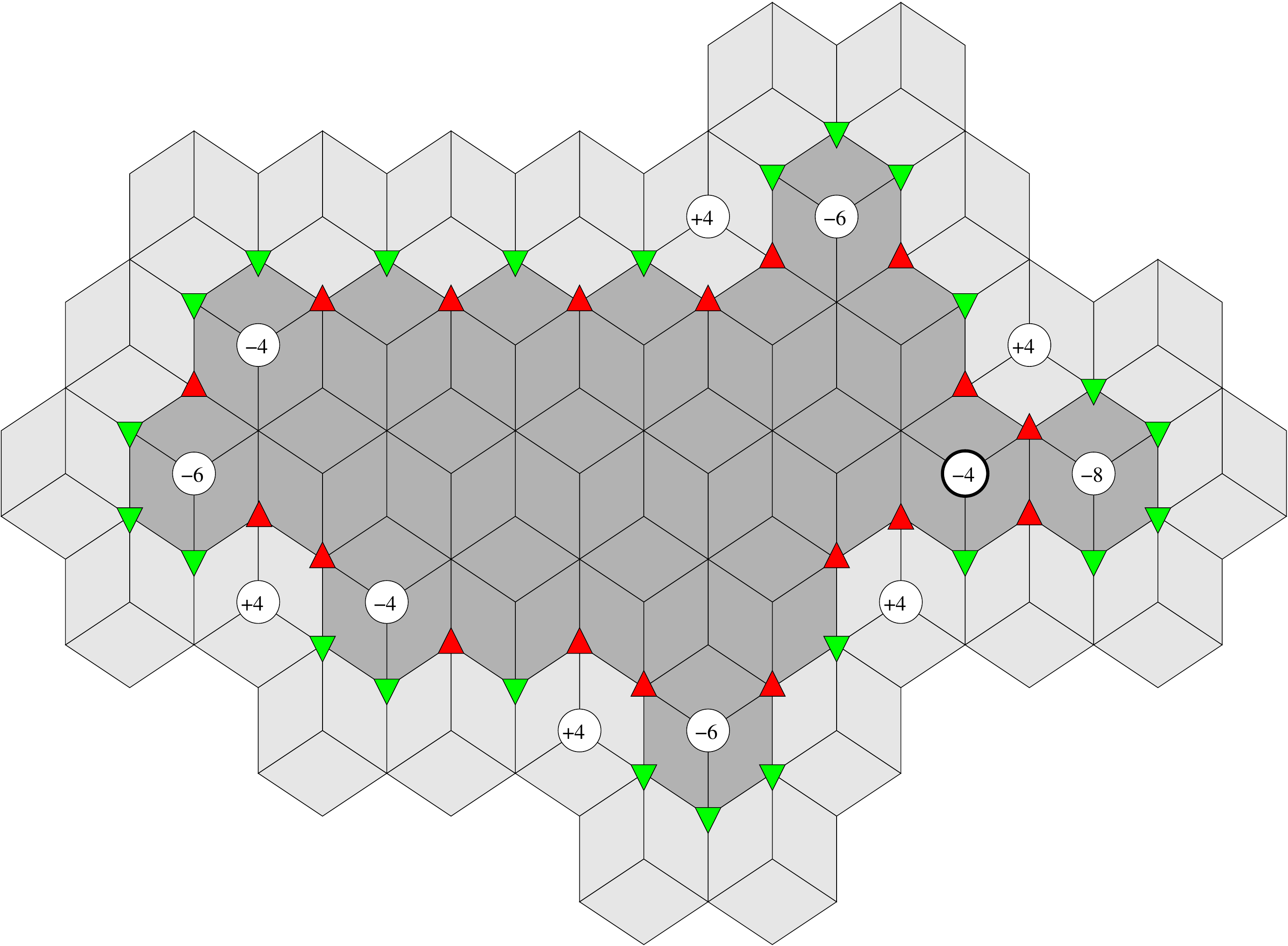}
\caption{A single triconvex island with $24$ salient angles (downwards triangles) and $18$ reflex angles (upwards triangles).
A number in a circle shows the variation of $\phi$ if a flip is performed around this circle.
There are $F(\omega)=12$ performeable flips, and one computes $\mathbb{E}(\Delta\phi)=-\frac{3}{2}\leq-\frac{12}{F(\omega)}$.
The bound is not tight because there is one flip (in the thick circle) which disconnects the island.
}
\label{fig:une_ile_convexe}
\end{figure}

Since islands are not always triconvex, we need the following technical lemma, which relies the variations $\Delta \overline{\phi}(\omega)$ and $\Delta \phi(\overline{\omega})$ (see Fig.~\ref{fig:une_ile}):

\begin{lemma}\label{lem:une_ile}
If $\omega$ is a dimer tiling with only one island, then
$$
F(\omega)\mathbb{E}[\Delta \overline{\phi}(\omega)]\leq F(\overline{\omega})\mathbb{E}[\Delta \phi(\overline{\omega})]
$$
\end{lemma}

\begin{proof}
Let $\Delta X_y$ denotes the $X$-variation by performing a flip around $y$.
Proving the lemma is then equivalent to prove
$$
\sum_{x\in F(\omega)}\Delta\overline{\phi}_x(\omega)\leq\sum_{x\in F(\overline{\omega})}\Delta \phi_x(\overline{\omega}).
$$
This easy follows from these three sublemmas:\\
$\bullet$ If $x\in F(\omega)\backslash F(\overline{\omega})$, then $\Delta\overline{\phi}_x(\omega)=0$ since the flip does not modify $\overline{\omega}$.\\
$\bullet$ If $x\in F(\overline{\omega})\backslash F(\omega)$, then $\Delta \phi_x(\overline{\omega})>0$.\\
We first prove that $\Delta V_x(\overline{\omega})=1$.
This is clear if $x$ is not in the island of $\overline{\omega}$.
Other\-wise, that is, if $x$ is in the island of $\overline{\omega}$, then it is on a segment whose endpoints are in the island of $\omega$ and whose height are less or equal to the one of $x$.
By performing the flip in $x$, these endpoints are unchanged.
Hence, the triconvexity of $\overline{\omega}$ ensures that the height of $x$ in $\overline{\omega}$ has increased, that is, $\Delta V_x(\overline{\omega})=1$.\\
We then prove that $\Delta E_x(\overline{\omega})=0$.
If $\Delta E_x(\overline{\omega})<0$, then $x$ is surrounded by at least $4$ error edges, among which at least two are parallel.
If $x$ is not in the island of $\overline{\omega}$, this would contradict the triconvexity of $\overline{\omega}$.
If $x$ is in the island of $\overline{\omega}$, thus in the island of $\omega$, then $x$ would be in $F(\omega)$, what contradicts our initial assumption.
Hence, $\Delta E_x(\overline{\omega})=0$ (since the cooling never increases the energy).\\
$\bullet$ If $x\in F(\omega)\cap F(\overline{\omega})$, then $\Delta\overline{\phi}_x(\omega)\leq\Delta \phi_x(\overline{\omega})$.
Indeed, the only problem is that performing this flip could change much more the triconvex hull of $\omega$ than $\omega$ itself.\\
Assume, first, that $\Delta E_x(\omega)=0$.
In this case, $x$ is surrounded by exactly three consecutive error edges (on the boundary of the six-triangle hexagon modified by the flip).
By performing the flip, we change these error edges by the three other ones.
In particular, to each new error corresponds a parallel old error, so that the the triconvex hull after and before the flip differ only in $x$.
Thus, $\Delta\overline{\phi}_x(\omega)=\Delta \phi_x(\overline{\omega})$.\\
Assume, now, that $\Delta E_x(\omega)<0$.
In this case, $x$ is surrounded by $i\geq 4$ error edges.
Since two of these edges are parallel, $x$ is in the island of $\omega$ (otherwise it would not be in $F(\overline{\omega})$).
One thus has $\Delta\phi_x(\overline{\omega})=-4-2(i-2)=-2i$.
If $x$ is still in the island of $\overline{\omega}$, then $\Delta \overline{\phi}_x(\omega)=0$.
Otherwise, once the flip performed, the new convex hull of the island has lost at least $x$, but maybe much more, so that we can just ensures that $\Delta \overline{\phi}_x(\omega)\leq -2i$.
This shows $\Delta\overline{\phi}_x(\omega)\leq\Delta \phi_x(\overline{\omega})$.
\qed\end{proof}

\begin{figure}
\centering
\includegraphics[scale=0.2]{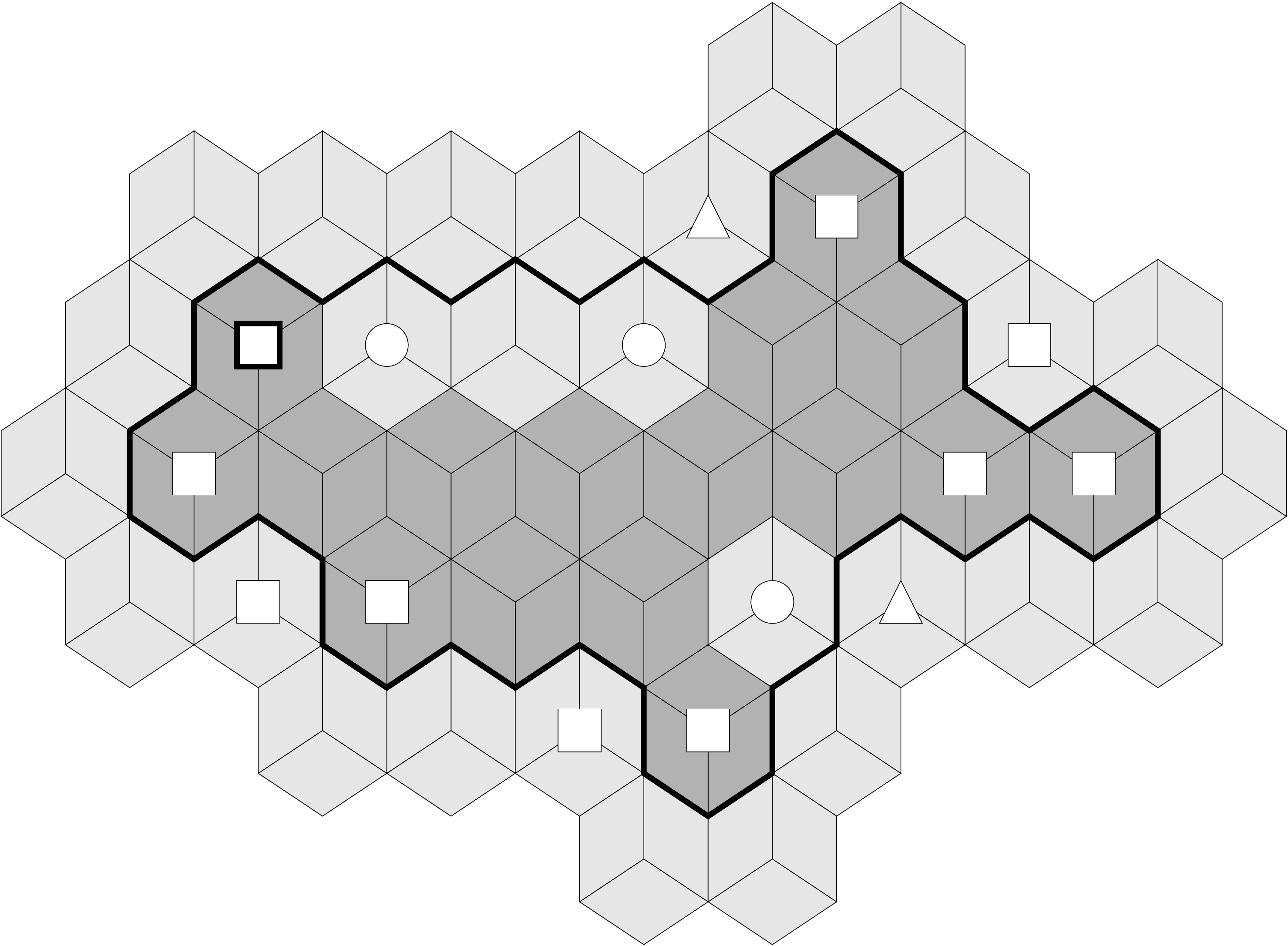}
\caption{An island of a tiling $\omega$.
Its triconvex hull (thick boundary) corresponds to Fig.~\ref{fig:une_ile_convexe}.
Flips are indicated by circles, triangles or squares, respectively depending whether they are in $F(\omega)\backslash F(\overline{\omega})$, $F(\overline{\omega})\backslash F(\omega)$ or $F(\omega)\cap F(\overline{\omega})$.
Lemmap~\ref{lem:une_ile} relies $\Delta \overline{\phi}(\omega)$ and $\Delta \phi(\overline{\omega})$.
}
\label{fig:une_ile}
\end{figure}

\noindent In particular, this result combined with Lem.~\ref{lem:une_ile_convexe} applied to $\overline{\omega}$ yield:
$$
\mathbb{E}[\Delta \overline{\phi}(\omega)]\leq
\mathbb{E}[\Delta \phi(\overline{\omega})]\leq-\frac{12}{F(\omega)}.
$$
Let us now extend this result to the case of several islands.
The only problem which can occurs is that flips can merge islands.
However, one big advantage of considering $\overline{\phi}$ instead of $\phi$ is that mergings becomes rather trivial, as we will see.

\begin{lemma}\label{lem:une_hauteur}
If $\omega$ is a dimer tiling whose islands have all the same height, then
$$
\mathbb{E}[\Delta \overline{\phi}(\omega)]\leq-\frac{12}{F(\omega)}.
$$
\end{lemma}

\begin{proof}
If there is no flip which can merge two or more islands of $\overline{\omega}$, then the result just follows from Lem.~\ref{lem:une_ile_convexe} and~\ref{lem:une_ile}.
Otherwise, consider such a flip.
It is thus performed around a vertex, say $x$, which is not in $\overline{\omega}$.
The triconvexity of $\overline{\omega}$ ensures that this is possible only if error and non-error edges alternate on the boundary of the six-triangle hexagon of center $x$.
This flip thus merges exactly three islands of $\overline{\omega}$.
Then, the triconvexity of $\overline{\omega}$ ensures that these islands are necessarily stick-shaped, that is, formed of aligned six-triangle hexagons (Fig.~\ref{fig:une_hauteur}, left).
At least the endpoints of these stick-shaped islands are in $\omega$, and one easily checks that the bound still holds (Fig.~\ref{fig:une_hauteur}, right).
\qed\end{proof}

\begin{figure}
\centering
\includegraphics[scale=0.2]{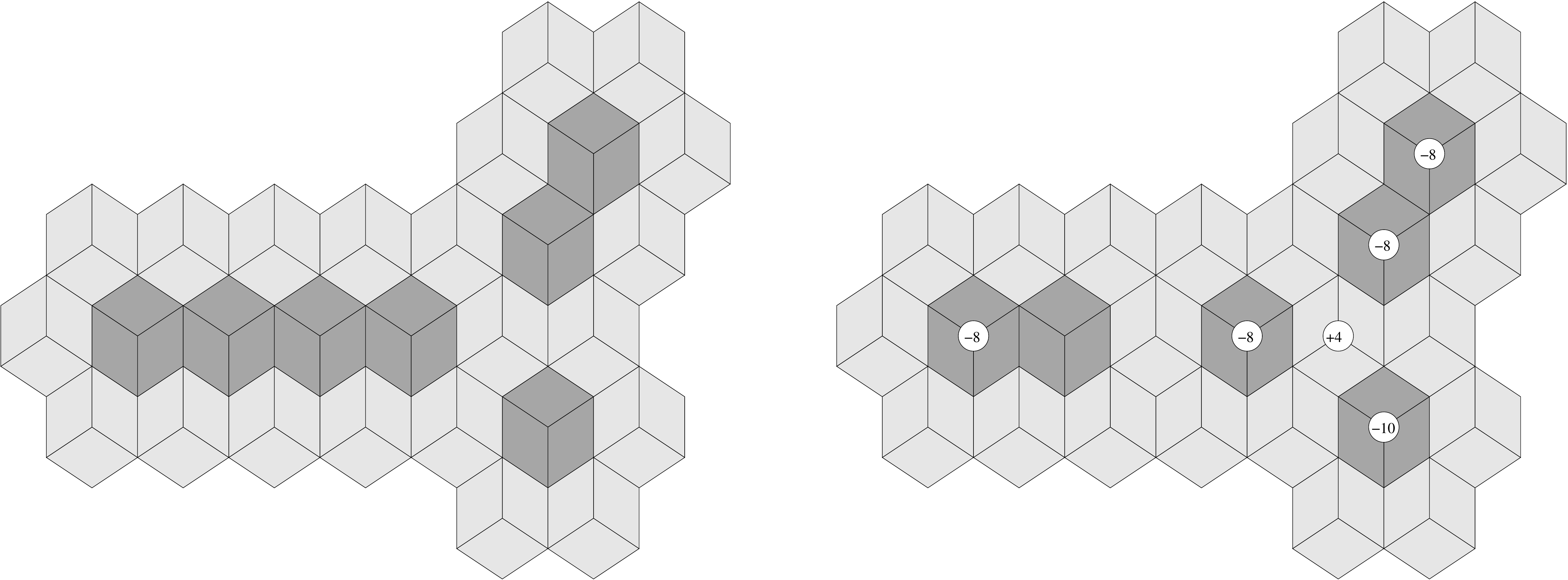}
\caption{A flip which merges islands of $\overline{\omega}$ necessarily merges exactly three islands, which are moreover stick-shaped (left: $\overline{\omega}$).
The flips in $\omega$ which modify $\overline{\omega}$ are such that $\phi$ decreases on expectation (right: $\omega$).}
\label{fig:une_hauteur}
\end{figure}

\noindent We are now in a position to prove:

\begin{theorem}\label{th:convergence}
The worst expected convergence time satisfies:
$$
\widehat{T}(n)=O(n^2\sqrt{n}).
$$
\end{theorem}

\begin{proof}
Let us consider the highest level of the tiling (that is, the highest islands).
The only problem to apply Lem.~\ref{lem:une_hauteur} is that some flips can be unperformeable, due to lower islands.
However, only flips which would increase the volume of $\overline{\omega}$ become unperformeable.
Moreover, these flips could not decrease the energy of $\overline{\omega}$ (we have already seen in the proof of Lem.~\ref{lem:une_ile} that a flip which increases the volume of a triconvex island cannot modify its energy).
The bound of Lem.~\ref{lem:une_hauteur} thus holds {\em a fortiori} for this highest level.
Since both the volume and the energy of the highest level of the tiling are at most $n$, and $F(\omega)\leq n$, one can apply Prop.~\ref{prop:probabilistic_tool1} to $\overline{\phi}$ with $\varepsilon=1/n$: this yields that the highest level of the tiling disappears in time $O(n^2)$.
The claimed bound then follows since the tiling has $O(\sqrt{n})$ levels.
\qed\end{proof}

\section{Prospects}
\label{sec:prospects}

\noindent {\bf Towards a tight bound}\\
Recall that numerical experiments led us to conjecture that the expected convergence time is in $O(n^2)$.
However, Th.~\ref{th:convergence} proves only a $O(n^2\sqrt{n})$ upper bound.
This means that we lost a factor $\sqrt{n}$ in the convergence time analysis.
We see at least two possibly large approximation in our proof, informally described below.\\

First, we have considered triconvex hulls of islands instead of real islands.
The main advantage is that we thus need to consider neither holes (hence we can suppose that an island has a simple boundary, as in Lem.~\ref{lem:une_ile_convexe}) not complicated merging of islands (the merging analysed in Lem.~\ref{lem:une_hauteur} are very simple).
The disadvantage is that, by considering triconvex hulls, we ignore the role of flips performed {\em inside} hulls (they just appear in the bound in the $F(\omega)$ term).
However, numerical experiments suggest that the cooling of a tiling is not significantly faster than the cooling of its triconvex hull.
Hence, it is not clear that we lost that much in considering triconvex hulls.\\

Second, in the proof of Th.~\ref{th:convergence}, we applied Lem.~\ref{lem:une_hauteur} only the highest level of the tiling.
This way, we make as if the $F(\omega)$ factor in the bound of Lem.~\ref{lem:une_hauteur} would count only flips on this level.
But, of course, $F(\omega)$ counts all the flips of the tiling.
In particular, one can reasonably expect that flips performed on lower levels act similarly, {\em i.e.}, that one could multiply the bound of Lem.~\ref{lem:une_hauteur} by the number of levels.
Since they are around $\sqrt{n}$ levels, this would greatly improve the bound.
However, the problem is that a flip which could be performed on some lower island if they were not upped islands, can not be any more performeable.
This would not be a problem if such flips could only increase $\overline{\phi}$, but we have examples where the converse occurs.
In other words, a lower island is forced to growth because of a higher island.
We can maybe get round this problem, since the higher island will rapidly shrink and no longer force the lower one to growth, but analyzing such a phenomenon seems to be very tedious\ldots\\

Last, let us point out that the simple theoretical upper bound we provided is $\Omega(n\sqrt{n})$, that is, we also need to gain a factor $\sqrt{n}$.\\

\noindent {\bf Average expected convergence time}\\
In this paper, we only consider the worst expected convergence time, $\widehat{T}(n)$.
However, it is also natural (even more, according to the physical motivations exposed in the introduction), to consider the average expected convergence time:
$$
\overline{T}(n):=\frac{1}{\#\Omega_n}\sum_{\omega\in\Omega_n}\mathbb{E}(T~|~\omega_0=\omega).
$$
In this case, numerical experiments (Fig.~\ref{fig:simulations2}) led us to conjecture:
$$
\overline{T}(n)=\Theta(n\sqrt{n}).
$$

\begin{figure}
\includegraphics[width=0.48\textwidth]{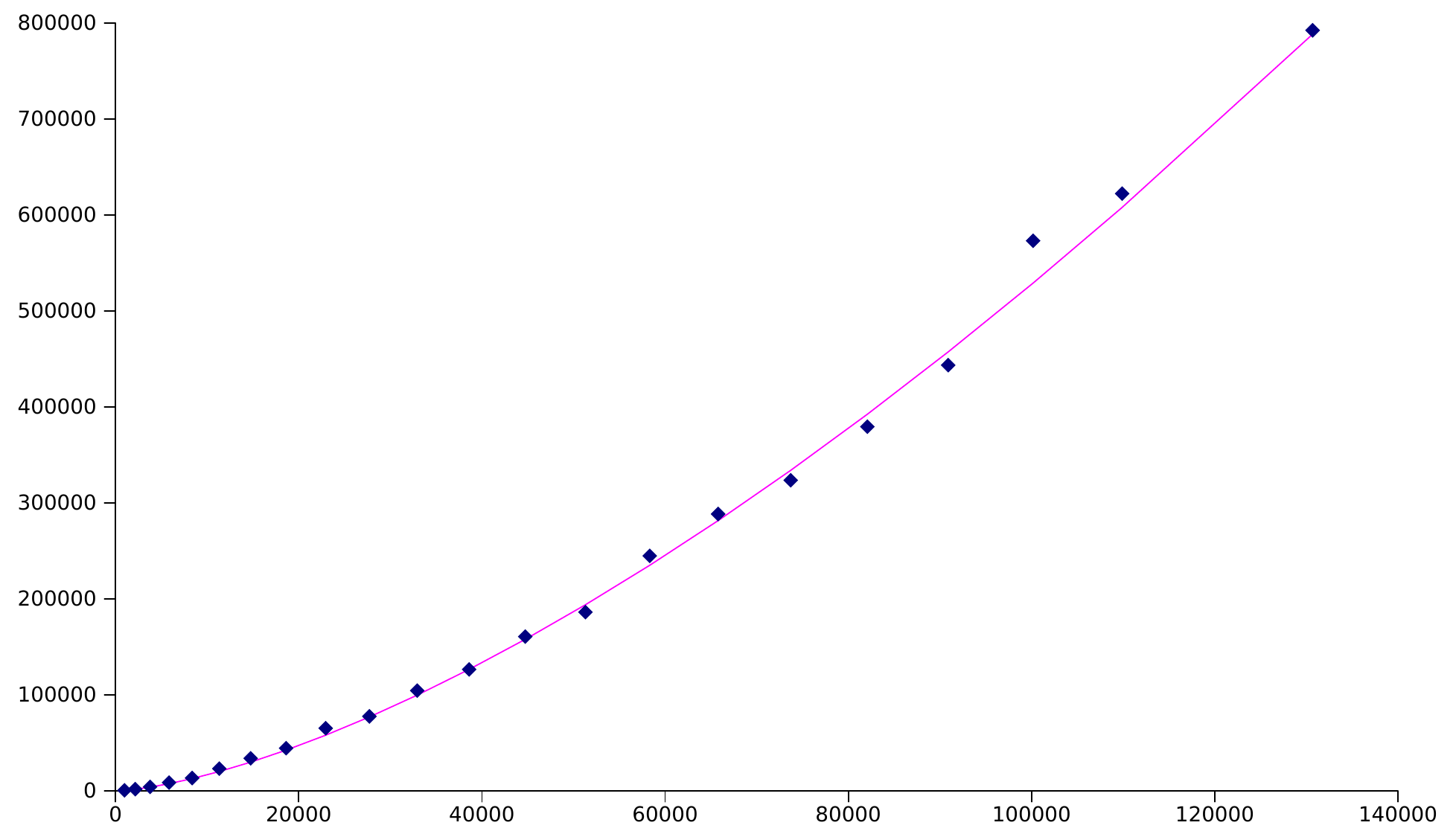}
\hfill
\includegraphics[width=0.48\textwidth]{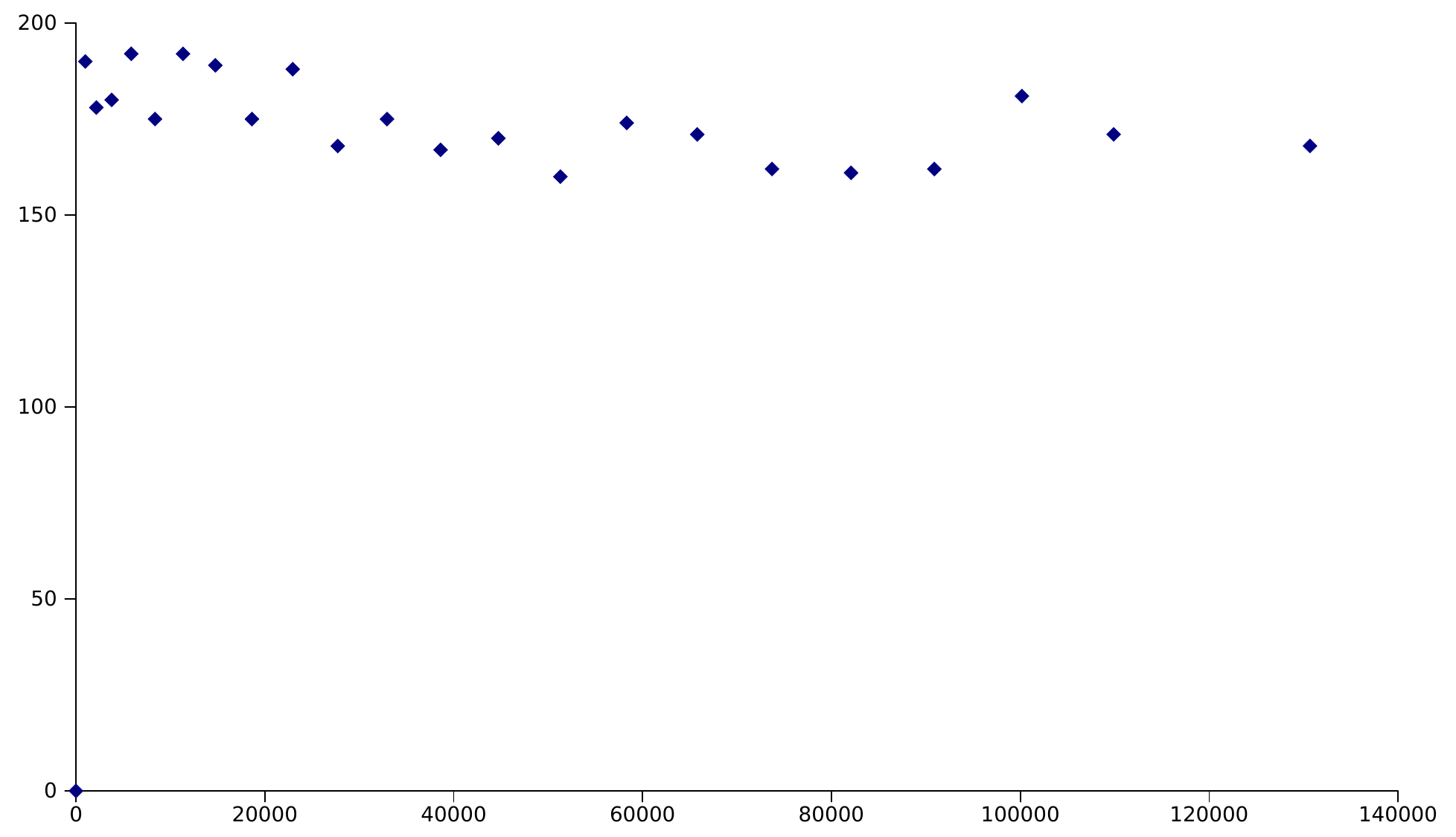}
\caption{Analog of Fig.~\ref{fig:simulations} for the average expected convergence time. Here, the drawn curve (left) has equation $cn\sqrt{n}$, $c=1.669.10^{-2}$.
We thus conjecture $\overline{T}(n)=\Theta(n\sqrt{n})$.
}
\label{fig:simulations2}
\end{figure}

In order to bound the average expected convergence time by similar technics as for the worst expected convergence time, we first need to compute the average values of some functions, {\em e.g.}, the volume $V$, the energy $E$ or the number of levels $H$.
This is what we did in the one-dimensional case in \cite{analco}.
However, this two-dimensional case is much harder, in the spirits of works by Propp \& al \cite{propp}.
Let us just mention, to conclude, that numerical experiments indicate that both the average volume and the average energy of a tiling are linear (in the number of tiles), while the number of levels is logarithmic!\\

\noindent {\bf Acknowledgments}. We would like to thank Olivier Bodini (LIP6, Paris), \'Eric R\'emila (LIP, Lyon) and Mathieu Sablik (LATP, Marseille) for useful discussions.


\end{document}